\documentclass[12pt,reqno]{amsart}
\usepackage{amssymb,enumerate}
\usepackage{amsmath}
\usepackage{graphics}

\usepackage{amssymb}
\usepackage{graphicx}
\usepackage{palatino}
\usepackage[a4paper,margin=1in]{geometry}

\newtheorem{theorem}{Theorem}
\newtheorem{lemma}{Lemma}
\newtheorem{corollary}{Corollary}
\newtheorem{proposition}{Proposition}
\newtheorem{definition}{Definition}

\usepackage{color}

\DeclareMathOperator*{\argmin}{arg\,min}
\DeclareMathOperator{\dom}{dom}
\DeclareMathOperator{\CAT}{CAT}

\begin{document}

\title{A product of strongly quasi-nonexpansive mappings in Hadamard spaces}
\author{Wiparat Worapitpong$^{1}$}
\author{Parin Chaipunya$^{2}$}
\author{Poom Kumam$^{3}$}
\author{Sakan Termkaew$^{4}$}
\address{$^{1}$
Department of Mathematics, Faculty of Science, King Mongkut's University of Technology Thonburi, 126 Pracha Uthit Rd., Bang Mod, Thung Khru, Bangkok, 10140, Thailand.}
\email{wiparat.w@mail.kmutt.ac.th}

\address{$^{2}$  Department of Mathematics, Faculty of Science, King Mongkut's University of Technology Thonburi, 126 Pracha Uthit Rd., Bang Mod, Thung Khru, Bangkok, 10140, Thailand.}
\email{parin.cha@kmutt.ac.th}

\address{$^{3}$  Center of Excellence in Theoretical and Computational Science (TaCS-CoE), King Mongkut's University of Technology Thonburi, 126 Pracha Uthit Rd., Bang Mod, Thung Khru, Bangkok, 10140, Thailand. }
\email{poom.kum@kmutt.ac.th}
 
\address{$^{4}$ Faculty of Learning Sciences and Education (Thammasat Secondary School), Thammasat University, 99 Moo 18 Paholyothin Road, Klong Nueng, Klong Luang, Pathumthani, 12121, Thailand.}
\email{sakan.te@lsed.tu.ac.th}


%

\begin{abstract}
In this paper, we prove that the product of strongly quasi-nonexpansive $\Delta$-demiclosed mappings is also a strongly quasi-nonexpansive orbital $\Delta$-demiclosed mapping in Hadamard spaces. Additionally, we establish the $\Delta$-convergence theorem for approximating a common fixed point of infinite products of these mappings in Hadamard spaces. Our results have practical applications in convex function minimization, the minimization of the sum of finitely many convex functions, and solving the convex feasibility problem for finitely many sets in Hadamard spaces.
\end{abstract}

\maketitle
\pagestyle{myheadings}
\markboth{Worapitpong, W.; Chaipunya, P.; Kumam, P.;  Termkaew, S.}{A product of strongly quasi-nonexpansive mappings in Hadamard spaces}
\section{Introduction}

Strongly quasi-nonexpansive mappings play crucial role in convex analysis and optimization. The notion of such mappings was initiated from the work in \cite{BruckR}, where the class of strongly nonexpansive mappings was analysed as a generalized class of firmly nonexpansive mappings. This strongly nonexpansive mappings resulted to the notion of strongly quasi-nonexpansive mappings in \cite{Bruck}. It is known (see, for example, \cite[Lemma 1.1]{Bruck}) that the composition of strongly quasi-nonexpansive will be again strongly quasi-nonexpansive on Banach space provided that the intersection of their sets of fixed points is nonempty.   

A substantial member of the class of strongly quasi-nonexpansive mappings is the resolvent operator for a proper lower semi-continuous function and that of maximal monotone mappings, which are essential in the development of proximal point algorithm; a well-known method for approximating a minimum point of a convex function or estimating zeros of monotone mappings \cite{Minty,Martinet,Rockafellar}. For the special case of indicator function over a closed convex set $C$, the associated resolvent operator coincides with the metric projection onto $C$, denoted by $P_C$. This operator is substantial in analysing several problems in convex analysis, particularly in solving convex feasibility problem in the sense of finding an element of certain intersection of closed convex sets as detailed in \cite{Bauschke,Cegielski2013}. 

Considering the significance of some strongly quasi-nonexpansive mappings in addressing several problems, particularly the utilization of resolvent or metric projection operators, these mappings are recently analysed beyond linear spaces, particularly to the extended settings of geodesic spaces where there is flexibility of accommodating nonconvex problems to convex setting in the sense of Alexendrov geometry. In particular, such spaces have proven to be substantial frameworks for the optimization analysis associated with the harmonic mappings as in \cite{Jost1}. For further details concerning this framework and its substantial geometry, refer to Bridson and Haefliger \cite{bridson1999metric}, Burago et. al \cite{burago2001course} and Jost \cite{Jost2}.

Composition of strongly quasi-nonexpansive mappings is highly substantial not only within the linear spaces but also in spaces  with Alexandrov geometry. For instance, Ba$\check{\text{c}}$\'{a}k, Searston, and Sims  proved  that the sequence $\lbrace (P_{A}P_{B})^n x \rbrace$ $\Delta$-converges to a point in $A \cap B$ whenever $A$ and $B$ are closed convex subsets of complete $\CAT(0)$ spaces (see \cite{BacakA}, Theorem 4.1). Moreover, the resolvent operator was analysed in complete CAT(0) spaces by Ba$\check{\text{c}}$\'{a}k  \cite{Bacak2013proximal} to solve convex optimization through some weak convergence form of the bounded sequence generated by the proximal point algorithm (Theorem 1.4 in \cite{Bacak2013proximal}). Later the apporoach is extended to the case when the function is the finite sum of proper lower semi-continuous functions (Theorem 6.3.1 in \cite{Bacak2014Convex}). This approach is applicable to solving problem associated with mean in metric phylogenetics tree \cite{Bacak2014Computing}. Moreover, the approach provide more substantial ground for the investigation of finite composition of strongly quasi-nonexpansive mappings. 

Recently, Ariza, L\'{o}pez, and Nicolae \cite{Ariza2015asymptotic} analyzed the asymptotic behaviour of compositions of finite firmly nonexpansive mappings through asymptotic regularity in the sense of Browder and Petryshyn \cite{BrowPet} under the uniform convexity assumptions. Infact, the authors obtained weak convergence of the iterates of the product of metric projections as the consequence of managing the product of a finite number of
firmly nonexpansive operators. Following this development,  Reich and Salinas \cite{Reich2016Weak} extended the work of Ba$\check{\text{c}}$\'{a}k, Searston, and Sims \cite{BacakA} to a countable product  of  firmly of metric projections and established weak convergence of the sequence generated therefrom.

In \cite{Khatibzadeh2020iterations}, the authors considered the sequence of strongly quasi-nonexpansive mappings in Hadamard spaces and provide convergence results for the sequence generated therefrom to a fixed point of the underlying operator. The authors established some substantial applications for the sequence of strongly quasi-nonexpansive mappings in utilization of fixed point iterative techniques, equilibrium problem, and minimization problem involving geodesically convex or  pseudoconvex functions. Thereafter, Termkeaw et al., \cite{termkaew2023infinite} studied the convergence of the product of a strongly quasi-nonexpansive orbital $\Delta$-demiclosed mapping and showed its applicability in convex optimization and convex feasibility problems. For further rescent development with regard to a closer notion of composition of family of strongly quasi-nonexpansive mappings in geodesic spaces, refer to \cite{Berdellima} and the references therein. 

Motivated by the aforementioned results, we establish certain properties for finite composition of strongly quasi-nonexpansive mappings, thereby deducing that composition of family of strongly quasi-nonexpansive mapping is also quasi-nonexpansive mapping and the set of fixed points of the resulted composed operator coincides with the intersection of the sets of fixed points of the individual mappings. Moreover, when each of the mapping is orbital $\Delta$-demiclosed the resulted composed mapping is also orbital $\Delta$-demiclosed. Furthermore, we establish
$\Delta$-convergence of the sequence of Picard iteration to a fixed point of the underlying operator when its orbital $\Delta$-demiclosed. This result is applicable to the common fixed point of finite family of such operators following the established properties for the finite composition. The results are utilize to solve some convex feasibility and common minimization problems in CAT(0) spaces.

The paper is organized as follows. In Section 2: Preliminaries, we present the background materials necessary for the establishment of the main results and describing the associated concepts. In Section 3: Main Results, we prove some useful properties regarding finite composition of strongly quasi-nonexpansive mappings and also establish the $\Delta$-convergence of the sequence of Picard iteration to a fixed point of the underlying mapping when its orbital $\Delta$-demiclosed mappings in Hadamard spaces. In Section 4: Applications, we explore potential applications of the result of results in Section 3 to solve some convex feasibility and common minimizer problems in Hadamard spaces. Finally, in Section 5: Conclusions, we provide the conclusion.

\section{Preliminaries}

In this section, some basic notions and useful lemmas necessary for the subsequent results are given. Throughout this paper, the set of all real numbers is denoted by $\mathbb{R}$.

Let $(X,d)$ be a metric space and $x,y \in X$. A {\it geodesic path} joining $x$ and $y$ is a mapping $\gamma : [0, 1]\rightarrow X$ such that $\gamma(0)=x$, $\gamma(1)=y$ and
$d(\gamma(t),\gamma(s))=d(x,y)|t-s|$ for any $t,s \in [0,1]$. We say that $X$ is a {\it (uniquely) geodesic metric space} if any two elements $x,y \in X$ are connected by a (unique) geodesic. We write $[x,y] := \gamma_{x,y}=\gamma([0,1])$ to denote the {\it geodesic segment} of $\gamma$ whenever $X$ is uniquely geodesic containing $x,y$, and $\gamma$ is the geodesic path joining $x$ to $y$. In this case, we also use the notation $(1-t) x\oplus t y := \gamma(t) = \gamma_{t}$ for any $t \in [0,1]$. 

In Uniquely geodesic space  $(X,d)$, a {\it geodesic triangle} with vertices $x,y,z \in X$, denoted by $\Delta (x,y,z)$, is defined by $[x,y] \cup [y,z] \cup [z,x]$. For a triangle $\Delta (x_1,x_2,x_3)$ {in a geodesic space} $X$, we can find the \textit{comparison triangle} $\Delta (\overline{x}_{1},\overline{x}_{2},\overline{x}_{3})$ in $\mathbb{E}^2$ such that $d(x_i,x_j) = \| \overline{x}_i - \overline{x}_j\|$ for all $i,j \in \{ 1,2,3 \}$. Let $\gamma$ be the geodesic joining $x_{i}$ to $x_{j}$ and $u := \gamma(t)$ for some $t \in [0,1]$. Then a point $\overline{u} := \overline{\gamma}(t)$ is called the {\it comparison point} of $u$, where $\overline{\gamma}: t\mapsto (1-t)\overline{x}_i + t \overline{x}_j$ is the geodesic joining $\overline{x}_{i}$ to $\overline{x}_{j}$. If every two points $p,q$ in $\Delta (x_1,x_2,x_3)$ and their comparison points $\overline{p},\overline{q}$ in $\Delta (\overline{x}_{1},\overline{x}_{2},\overline{x}_{3})$ satisfy 
$$d(p,q)\leq \| \overline{p} - \overline{q} \|,$$
then $X$ is called a $\CAT(0)$ space. A complete $\CAT(0)$ space is called a \textit{Hadamard space}.

Let $(X,d)$ be a uniquely geodesic metric space and $ f : X \rightarrow (-\infty,+\infty]$. A subset $C\subset X$ is called {\it convex} if $[x,y] \subset C$ for all $x,y \in C$. The  {\it effective domain} of $f$ is defined by $\dom f := \{ x \in X \,|\, f(x) < +\infty \}$. If $\dom f$ is nonempty, we say that $f$ is \textit{proper}. A function $f$ is said to be \textit{lower semicontinuous} (shortly, lsc) if $f(u) \leq \liminf_{n \rightarrow \infty} f(u_{n})$ whenever $u \in X$ and $\lim_{n \rightarrow \infty} u_{n} = u$. In addition, $f$ is said to be a {\it convex} if 
\[
f((1-t)x\oplus ty) \leq (1-t)f(x) + t f(y)
\]
holds for all $x,y \in X$ and $t \in (0,1)$.

Let $(X,d)$ be a Hadamard space and $\lbrace x_n \rbrace$ be a bounded sequence of $X$. The \textit{asymptotic center} $\mathcal{A} ( \lbrace x_n \rbrace )$ of sequence $\lbrace x_n \rbrace$ is defined by
$$ \mathcal{A} ( \lbrace x_n \rbrace ) = \left\lbrace z \in X \, | \, \limsup_{n \to \infty} d(z,x_n) = \underset{y \in X}{\inf} \limsup_{n \to \infty} d(y,x_n) \right\rbrace. $$
A well-known result implies that $\mathcal{A} ( \lbrace x_n \rbrace )$ is a singleton in Hadamard spaces \cite[Proposition 7]{Dhompongsa2006}. 
The sequence $\lbrace x_n \rbrace$ is said to be \textit{$\Delta$-convergent} to an element $z$ of $X$ if
$$ \mathcal{A} ( \lbrace x_{n_i} \rbrace ) = \lbrace z \rbrace, $$
for each subsequence $\lbrace x_{n_i} \rbrace $ of $ \lbrace x_n \rbrace $. In this case, we say that $z$ is the \textit{$\Delta$-limit} of $\lbrace x_n \rbrace$. If the sequence $\lbrace x_n \rbrace$ is convergent to $z$, then it is $\Delta$-convergent to $z$. It is worth noting that the definition of $\Delta$-convergence is given for bounded sequences. Moreover, if the sequence $\lbrace x_n \rbrace$ is $\Delta$-convergent to $z$, then every subsequence  $\lbrace x_{n_i} \rbrace$ of $\lbrace x_n \rbrace$ is $\Delta$-convergent to $z$ \cite{kirk2008concept}. This notion of $\Delta$-convergence is attributed to 
Lim \cite{LimRemark}, initially in general metric spaces, and later specialized to CAT(0) settings by {Kirk and Panyanak \cite{kirk2008concept}}. It turns out that several properties of weak convergence are extended to the $\Delta$-convergence settings in CAT(0) spaces. A subset $C \subset X$ is called \textit{$\Delta$-closed} if $y \in C$ whenever $y$ is a $\Delta$-limit of some sequence $\lbrace y_n \rbrace$ of $C$. In the sequel, we use $\omega_{\Delta}(\lbrace x_n \rbrace)$ to denote the set of all $\bar{z} \in X$ whenever $\bar{z}$ is a $\Delta$-limit of some subsequence $\lbrace x_{n_i} \rbrace$ of $\lbrace x_n \rbrace$.

Let $(X, d)$ be a Hadamard space and $C$ be a nonempty, closed, and convex subset of $X$. Suppose that $S: C \to C$ is a mapping and $Fix(S):= \bigl\{ x \in C | S(x) = x\bigr\}.$  The mapping $S$ is said to be nonexpansive if $ d\big(S(x), S(y)\big) \leq d(x, y),$ for all $ x, y \in C$. The mapping $S$ is said to be quasi-nonexpansive if $Fix(S) \neq \emptyset$ and $d(S(x), q) \leq d(x,q),$ for any $x \in C,$ and $q \in Fix(S).$ 
The mapping $S$ is called strongly quasi-nonexpansive if $S$ is quasi-nonexpansive and $\lim\limits_{n \to \infty} d\left(x_n, S(x_n)\right) =0$, whenever $\{ x_n\}$ is a bounded sequence in $C$ such that $\lim\limits_{n \to \infty} \Bigl[d(x_n,q) - d\left(S(x_n),q\right) \Bigr]= 0$, for some $q \in Fix(S)$.  Moreover, $S$ is said to be $\Delta$-demiclosed in the sense of \cite{Kimura} if for any $\Delta$-convergent sequence $\lbrace x_n \rbrace$ in $X$, its $\Delta$-limit belongs to $Fix(S)$ whenever $\underset{n \rightarrow \infty}{\lim} d(x_n,S(x_n)) = 0$.

Let $(X,d)$ be a Hadamard space and $C$ be a nonempty closed convex subset of $X$. The metric projection $P_C : X \rightarrow C$ is defined by 
$$P_C(x)= \underset{y \in C}{ \argmin } \, d(x,y),$$ 
for every $x\in X$. Then for all $x \in X$, there exists a unique $\bar{x} \in C$ such that $P_C (x) = \bar{x}$. It is known that $Fix(P_C) =C$ and
$$ d(y,P_C x) \leq d(y,x), $$
for every $y \in C$ and $x \in X$ (see \cite[Proposition 3.5]{Espinola2009410}). In particular, the metric projection onto a nonempty closed and convex subset of a Hadamard space is strongly quasi-nonexpansive is $\Delta$-demiclosed (see \cite{Espinola2009410,Kimura,Ariza}).

Following the works of Jost \cite{Jost1,Jost2} and Mayer \cite{Mayer1998Gradient}, the resolvent $R_{\lambda,f}$ of a proper lower semicontinuous convex function $f$ of a Hadamard space $X$ into $(-\infty,\infty]$ with respect to $\lambda > 0$ is defined by
\[
R_{\lambda, f} (x) = \argmin_{y \in X} \left\lbrace f(y) + \frac{1}{2 \lambda} d(x,y)^2 \right\rbrace,
\]
for all $x \in X$. This mapping is well-defined as a single-valued mapping. It is a strongly quasi-nonexpansive that is $\Delta$-demiclosed mapping such that $Fix(R_{\lambda,f}) = \argmin_{X} f$ (see \cite{Khatibzadeh2020iterations,Espinola2009410,Ariza}).

Recently, Termkeaw et al. \cite{termkaew2023infinite} defined a new mapping, which is a generalization of $\Delta$-demiclosed in $\CAT (1)$ space. We particularized this notion to Hadamard spaces serving as a generalization of $\Delta$-demiclosedness.

\begin{definition}
	Let $(X,d)$ be a Hadamard space and let $S$ be a mapping from $X$ to $X$. The mapping $S$ is called {orbital $\Delta$-demiclosed} if $x^\ast \in$ {$Fix(S)$} whenever $\lbrace x_k \rbrace$ is a sequence in $X$ satisfying the following conditions:
	\begin{enumerate}
		\item $x_k \overset{\Delta}{\rightarrow} x^*$
		\item $d(S(x_k) , x_k ) \rightarrow 0$ and
		\item $(x_k)$ is a subsequence of $\lbrace  S^n (\hat{x}) \rbrace$ for some $\hat{x} \in X$.
	\end{enumerate}
\end{definition}
It is clear to see that every $\Delta$-demiclosed operator is also orbital $\Delta$-demiclosed.\\

The following lemmas are important for our main results. 

\begin{lemma}\cite{kirk2008concept} \label{kirk-lemma} 
     Let $(X,d)$ be a Hadamard space. Every bounded sequence in $X$ has a $\Delta$-convergent subsequence.
\end{lemma}

\begin{lemma}\cite{dhompongsa2008convergence}  \label{dhompongsa-lemma}
    Let $(X,d)$ be a Hadamard space. If  $\{x_n\}$ is a bounded sequence in $X$ with $ \mathcal{A}(\{x_n\}) = \{x\}$ and $\{u_n\}$ is a subsequence of $\{x_n\}$ with $\mathcal{A}(\{u_n\}) = \{ u \}$ and the sequence $\{d(x_n, u) \}$ converges, then $x = u.$
\end{lemma}

\section{Main result}

In this section, we establish that finite product of strongly quasi-nonexpansive mappings is a strongly  quasi-nonexpansive, and the set of fixed points of the resulted mapping coincides with the intersection of the sets of fixed points of the individual mappings. Moreover, when each of the mapping is $\Delta$-demiclosed the resulted mapping is orbital $\Delta$-demiclosed. Additionally, we prove
$\Delta$-convergence of the sequence of Picard iteration to a fixed point of the underlying operator when its orbital $\Delta$-demiclosed. This allows the application of common fixed point of finite family of such mappings. Firstly, we introduce the following lemmas, which serve as crucial tools in our main results. 

\begin{lemma} \label{lem1}
Let $(X, d)$ be a Hadamard space. If $S_i : X \to X$ is a strongly quasi-nonexpansive mapping for all $i \in \{1, . . . ,m\}$ and $\bigcap\limits_{i=1}^m Fix(S_i) \neq \emptyset,$ then \[Fix(S_m S_{m-1} \cdots S_1) = \bigcap\limits_{i=1}^m Fix(S_i).\]
\end{lemma}

\begin{proof}
First, to show $\bigcap\limits_{i=1}^m Fix(S_i) \subset Fix(S_m S_{m-1} \cdots S_1),$  let $p \in \bigcap\limits_{i=1}^m Fix(S_i)$. 
It follow that 
\begin{align*}
S_m S_{m-1} \cdots S_1(p)& = S_m S_{m-1} \cdots S_2(p)\\
&=S_m S_{m-1} \cdots S_3(p)\\
&~\vdots \\
&= S_m(p)=p
\end{align*}
Thus, we obtain  $p \in Fix(S_m S_{m-1} \cdots S_1).$  Therefore $ \bigcap\limits_{i=1}^m Fix(S_i) \subset Fix(S_m S_{m-1} \cdots S_1). $ 

For the converse, we proceed to prove 
\begin{equation}\label{eq.lem3.3}
Fix(S_m S_{m-1} \cdots S_1) \subset \bigcap\limits_{i=1}^m Fix(S_i) 
\end{equation} 
by induction. For the case of $m=2$, let $p \in Fix(S_2S_1)$ and $S_1(p)=x \in X$.  Then $p = S_2S_1(p) = S_2(x).$
As $S_1, S_2$ are quasi-nonexpansive, then for $q \in Fix(S_1)\cap Fix(S_2)$, we have \[d(p,q)=d(S_2(x),q) \leq d(x,q) \quad \text{and}\quad d(x,q)= d(S_1(p),q) \leq d(p,q).\] Thus, $d(S_2(x),q)=d(x,q).$ 
Then, $\lim_{n \to \infty} \bigl[ d(S_2(x),q) - d(x,q)\bigr] = 0.$ Since $S_2$ is strongly quasi-nonexpansive, so we have 
\[ \lim_{n \to \infty} d(S_2(x),x)=0  \quad \text{and} \quad S_2(x) = x.\]
Therefore \,$p=S_2S_1(p)=S_2(x)=x=S_1(p).$ 
This mean that $p \in Fix(S_1)$ and $p \in Fix(S_2)$. Thus, we have $Fix(S_2S_1) \subset Fix~(S_1) \cap Fix(S_2).$

Next, suppose that the equation (\ref{eq.lem3.3}) holds for case $m-1$. 
Let $p \in Fix(S_m \cdots S_1)$ and $S_{m-1} \cdots S_1(p)=:x \in X.$ Then, 
\[S_m(S_{m-1} \cdots S_1(p)) = S_m (x) = p.\] 
Since $S_i$ is quasi-nonexpansive, for all $i = \{1, ..., m\}$, then for $q \in \bigcap\limits_{i=1}^m Fix(S_i)$, we have  
\[
d(p,q)=d(S_m(x),q) \leq d(x,q) = d(S_{m-1} \cdots S_1(p),q),
\]
which $d(S_{m-1} \cdots S_1(p),q) \leq d(S_{m-2} \cdots S_1(p),q) \leq \cdots \leq d(S_1(p),q) \leq d(p,q).$ Hence $d(S_m(x),q)=d(x,q).$ Thus, $\lim_{n \to \infty} \bigl[ d(S_m(x),q) - d(x,q)\bigr]=0.$ 
Since $S_m$ is strongly quasi-nonexpansive. Then, \[\lim_{n \to \infty}d(S_m(x),x) = d(S_m(x),x) = 0, \, S_m(x)=x\] and 
\[ p=S_m\cdots S_1(p)=S_m(x)=x=S_{m-1}\cdots S_1(p) .\] 
This implies that $p \in Fix(S_m)$ and $p \in Fix(S_{m-1} \cdots S_1)$. By induction holds for $m-1$, we have $\bigcap\limits_{i=1}^{m-1} Fix(S_i)$. 
Therefore, $Fix(S_{m}  \cdots S_1) \subset \bigcap\limits_{i=1}^{m} Fix(S_i).$ This completes the prove.
\end{proof}

\vspace{2mm}

\begin{lemma} \label{lem2}
Let $(X, d)$ be a Hadamard space. If $S_i : X \to X$ is a strongly quasi-nonexpansive mapping for all $i \in \{1, . . . ,m\}$ and $\bigcap\limits_{i=1}^m Fix(S_i) \neq \emptyset.$ Then, $S=S_m S_{m-1} \cdots S_1$ is a strongly quasi-nonexpansive mapping.
\end{lemma}

\begin{proof}
We proceed to prove this lemma by mathematical induction. For the case of $m=2$, assume that $S_1$ and $S_2$ are strongly quasi-nonexpansive mapping such that $Fix(S_1) \cap Fix(S_2) \neq \emptyset.$   
Let $S:=S_2 S_1$, by Lemma \ref{lem1}, we have $Fix(S)=Fix(S_2S_1)= Fix(S_1)\cap Fix(S_2)$. Let $\{x_n\}$ be a bounded sequence in $X$ such that $\lim\limits_{n \to \infty} \bigl[d(x_n,y) - d\left(S(x_n),y\right) \bigr]= 0$, for all $y \in Fix(S)$. 

The first part is to show that $S$ is quasi-nonexpansive. 
We take $x \in X$ and $y \in Fix(S).$ By Lemma \ref{lem1}, $y \in Fix(S_1)\cap Fix(S_2)$, we have 
\[0 \leq d(S(x),y)=d(S_2S_1(x),y) \leq  d(S_1(x),y) \leq d(x,y).\]
That is $d(S(x),y) \leq d(x,y)$. Thus, $S$ is quasi-nonexpansive mappings. \\

The second part is to show that $S=S_m S_{m-1} \cdots S_1$ is strongly quasi-nonexpansive. \\
Since $S_1$ and $S_2$ are strongly quasi-nonexpansive, we obtain 
\begin{equation}\label{eq1} 
   0 \leq d(S(x_n),y)=d(S_2S_1(x_n),y) \leq  d(S_1(x_n),y) \leq d(x_n,y) < \infty. 
\end{equation}
Moreover, we have  $-d(x_n,y) \leq - d(S_1(x_n),y) \leq -d(S(x_n),y)$. Then, we get \[0=d(x_n,y)-d(x_n,y) \leq d(x_n,y)- d(S_1(x_n),y) \leq d(x_n,y)-d(S(x_n),y).\]
Take limit as $n \to \infty$, then we obtain $\lim\limits_{n \to \infty} \bigl[d(x_n,y) - d\left(S_1(x_n),y\right) \bigr]= 0.$ 
\\
Next, we add $- d(S(x_n),y)$ in to inequality (\ref{eq1}); we obtain \[0 \leq d(S_1(x_n),y)-d(S(x_n),y) \leq d(x_n,y)-d(S(x_n),y).\] 
Take limit as $n \to \infty$, we have that $\lim\limits_{n \to \infty} \bigl[d(S_1(x_n),y) - d\left(S(x_n),y \right) \bigr]= 0$, that is 
\begin{equation}\label{eq3.3}
\lim\limits_{n \to \infty} \Bigl[d(S_1(x_n),y) - d\big(S_2S_1(x_n),y \big) \Bigr]= 0. 
\end{equation}
From inequality (\ref{eq1}), $d(S_1(x_n),y) \leq d(x_n,y)$ and $\{x_n\}$ is bounded, it follows that $\{ S_1(x_n) \}$ is bounded.
Then by the triangle inequality, we have 
\begin{equation}\label{eq3.4}
0 \leq d(S(x_n),x_n)=d(S_2S_1(x_n),x_n)\leq d(S_2S_1(x_n),S_1(x_n)) +d(S_1(x_n),x_n). 
\end{equation}

Since $\{x_n\}$ is bounded and by inequality (\ref{eq1}), then $\lim_{n \to \infty}d(S_1(x_n),x_n) = 0.$  
Since $\{S_1(x_n)\}$ is bounded and by inequality (\ref{eq3.3}), then $\lim_{n \to \infty}d(S_2(S_1(x_n)),S_1(x_n)) = 0.$  \,
Take limit as $n \to \infty$ to inequality (\ref{eq3.4}), we obtain $\lim\limits_{n \to \infty} d\bigl(S(x_n),x_n\bigr) = 0$. Therefore $S=S_2S_1$ is strongly quasi-nonexpansive. \\

Next, for case $m>2$, suppose that $S_i$ is quasi-nonexpansive mapping for $i =\{1, ..., m\}$ and $ \bigcap\limits_{i=1}^m Fix(S_i) \neq \emptyset.$ 
Assume that $S_i$ is strongly quasi-nonexpansive mapping for all $i =\{1, ..., m\}$, and the product $S_{m-1} \cdots S_1$ is strongly quasi-nonexpansive. 
Let $S:=S_{m} \cdots S_1$ and $\{ x_n\}$ be a bounded sequence such that $\lim\limits_{n \to \infty} \bigl[d(x_n,y) - d\left(S(x_n),y \right) \bigr]= 0$, for some $y \in Fix(S).$ 

The first part of the case $m >2$\, is to show that  $S=S_m S_{m-1} \cdots S_1$ is quasi-nonexpansive mapping. 
For $x \in X$ and $y \in Fix(S)$, by Lemma \ref{lem1}, $Fix(S)=Fix(S_m S_{m-1} \cdots S_1)=  \bigcap\limits_{i=1}^{m} Fix(S_i).$ Then we have $ d(S(x),y) \leq d(x,y)$. Thus, the composition $S=S_m S_{m-1} \cdots S_1$ is quasi-nonexpansive mappings.

The second part is to show that $S$ is strongly quasi-nonexpansive, that is to shows that $\lim\limits_{n \to \infty} d\bigl(S(x_n),x_n\bigr) = 0$. 
Since $S_i$ is quasi-nonexpansive for each $i=\{1, ... , m\}$, and $\{ x_n\}$ is a bounded sequence, we have  
 \begin{align}\label{eq2}
   0 \, \leq \, d(S(x_n),y) & = \, d(S_m S_{m-1} \cdots S_1(x_n),y) \nonumber\\
                              & \leq  \, d(S_{m-1} \cdots S_1(x_n),y) \, \leq  \, \cdots \leq d(S_1(x_n),y) \nonumber\\
                              & \leq  \, d(x_n,y). 
 \end{align}
Moreover, we obtain  $-d(x_n,y) \leq - d(S_1(x_n),y) \leq \cdots
 \leq -d(S(x_n),y)$. Then, we get 
 \begin{eqnarray}
    0=d(x_n,y)-d(x_n,y)  & \leq & d(x_n,y)- d(S_1(x_n),y)\nonumber\\
    &\leq & d(x_n,y)-d(S_2S_1(x_n),y).\nonumber\\
    &\vdots & \nonumber\\
    &\leq & d(x_n,y)-d(S_{m-1} \cdots S_1(x_n),y)\nonumber\\
    &\leq & d(x_n,y)-d(S_m \cdots S_1(x_n),y)\nonumber\\
    &= & d(x_n,y)-d(S(x_n),y). 
\end{eqnarray}
Take limit as $n \to \infty$, and\, $\lim\limits_{n \to \infty} \bigl[d(x_n,y) - d\left(S(x_n),y \right) \bigr]= 0$.  Then, we obtain 
\[\lim\limits_{n \to \infty} \bigl[d(x_n,y) - d\left(S_{k} \cdots S_1(x_n),y \right) \bigr]= 0,\] 
for all~$k = 1,2,\dots,m$.
Since $S_i$ is strongly quasi-nonexpansive for each $i=\{1,...,m\}$ and $\{ x_n\}$ is bounded, it follow that $\{ S_i(x_n) \}$ is bounded for all $i= \{1, ..., m\}$. 
Then by the triangle inequality, we have 
\begin{eqnarray}
  \quad 0 &\leq& d(S(x_n),x_n) \nonumber\\
  &=& d(S_m\cdots S_1(x_n),x_n)\nonumber\\
  &\leq & d\left( S_m S_{m-1} \cdots S_1(x_n), S_{m-1} \cdots S_1(x_n)\right)\nonumber\\
  & & + d\left( S_{m-1} S_{m-2}\cdots S_1(x_n),S_{m-2} \cdots S_1(x_n)\right) + \cdots \nonumber\\
  & & + d\left( S_2 S_1(x_n),S_1(x_n) \right)+d( S_1(x_n), x_n )
\end{eqnarray}
Take limit as \,$n \to \infty$, and\, by the limits above, we obtain $\lim\limits_{n \to \infty} d\bigl(S(x_n),x_n\bigr) = 0$. Hence, $S=S_m S_{m-1} \cdots S_1$  is strongly quasi-nonexpansive. This completes the prove. 
\end{proof}

\vspace{3mm}
Now, we are ready to prove the main result.

\begin{proposition}  \label{prop1}
Let $(X, d)$ be a Hadamard space.  If $ S_i : X \to X $ is a strongly quasi-nonexpansive $\Delta$-demiclosed mapping for all $ i \in \{1, 2, . . . ,m\}$  and $\bigcap\limits_{i=1}^m Fix(S_i) \neq \emptyset.$  Then, $ S=S_m S_{m-1} \cdots S_1 $ is a strongly quasi-nonexpansive orbital $\Delta$-demiclosed mapping.
\end{proposition}

\begin{proof}
Suppose that $ S_i$ is a strongly quasi-nonexpansive  $\Delta$-demiclosed mapping for all $i \in \{1, 2, . . . ,m\}$ and $\bigcap\limits_{i=1}^m Fix(S_i) \neq \emptyset.$ We denote $ S=S_m S_{m-1} \cdots S_1 $, by Lemma \ref{lem2} we obtain $ S $ is a strongly quasi-nonexpansive mapping. Thus the remaining is to show that $S$ is orbital $\Delta$-demiclosed.

Since $ S_i$ is $\Delta$-demiclosed mapping for all $i \in \{1, 2, . . ., m\}$. Let $\{ x_n \}$ be a sequence in $X$ and $p$ be an element of $X$ such that 
\begin{enumerate}
    \item [(i)] $x_n \xrightarrow{\Delta} p$;
    \item [(ii)] $d(S(x_n),x_n) \to 0$;
    \item [(iii)] $\{ x_n\}$ is a subsequence $\{ S^{k_n}(z)\}$ for some $z \in X.$
\end{enumerate}

To claim that $S$ is orbital $\Delta$-demiclosed, we need to show that $p \in Fix(S)$. Since $S$ is strongly quasi-nonexpansive, then $Fix(S) = \bigcap\limits_{i=1}^m Fix(S_i) $ by Lemma \ref{lem1}, and $\bigcap\limits_{i=1}^m Fix(S_i)  \neq \emptyset$ by assumption. 
Take $y \in Fix(S).$  
Since $S$ is strongly quasi-nonexpansive by Lemma \ref{lem2}, we obtain  $d(S(x_n),y) \leq d(x_n,y). $  
By the triangle inequality, we have $d \bigl(x_n,y\bigr) \leq d\bigl(x_n, S(x_n)\bigr) + d\bigl(S(x_n),y\bigr). $  
Thus, we obtain $d \bigl(S(x_n),y\bigr) \leq  d\bigl(x_n,y\bigr) \leq d\bigl(x_n, S(x_n)\bigr) + d\bigl(S(x_n),y\bigr).$ 
Then $0 \leq d(x_n,y) - d\bigl(S(x_n),y\bigr) \leq d\bigl(x_n, S(x_n)\bigr).$  Take the limit as $n \to \infty$, by condition (ii) and the Squeeze Theorem, 
\[\lim\limits_{n \to \infty} \Bigl[d(x_n,y) - d\bigl(S(x_n),y\bigr) \Bigr]= 0. \]

Now, Using the condition (iii), there exists $z \in X$ such that $x_n = S^{k_n}(z)$, which gives 
\begin{equation} \label{eq4}
d(x_n,y) = d\bigl(S^{k_n}(z),y\bigr) \leq d\bigl(S^{k_n-1}(z),y\bigr) \leq \cdots \leq d(z, y) < \infty.   
\end{equation}
This inequality gives $d(x_n,y) \leq d(z,y)$. 
Since the sequence $\{ x_n\}$ is bounded sequence, then there exists a subsequence $\{ x_{n_j}\}$ of $\{ x_n\}$,  such that 
\begin{equation}\label{eq-P6}
\lim\limits_{j \to \infty} d(x_{n_j},y)  = \lim\limits_{j \to \infty} d\bigl(S(x_{n_j}),y \bigr).
\end{equation}
Since $S$ is strongly quasi-nonexpansive and $Fix(S) = \bigcap\limits_{i=1}^m Fix(S_i) $ by Lemma \ref{lem1}. Then we obtain 
\[d\bigl(S(x_{n_j}),y\bigr) \leq d\bigl(S_{m-1}S_{m-2}\cdots S_1 (x_{n_j}),y\bigr) \leq \cdots \leq d\bigl(S_1(x_{n_j}),y\bigr) \leq d(x_{n_j},y) \]
Take limit as $ j \to \infty$ and by equation (\ref{eq-P6}), we have 
\begin{eqnarray}\label{eq6}
    \lim\limits_{j \to \infty} d\bigl(S(x_{n_j}),y\bigr)  &  = &  \lim\limits_{j \to \infty} d\bigl(S_{m-1} \cdots  S_1(x_{n_j}),y\bigr)\nonumber\\
    &\vdots & \nonumber\\
    &= & \lim\limits_{j \to \infty} d\bigl(S_1(x_{n_j}),y\bigr) \nonumber\\
    &= & \lim\limits_{j \to \infty} d (x_{n_j},y ). 
\end{eqnarray}


 For case $i=1$, we will show that $p \in Fix(S_1).$ From equation (\ref{eq6}), we have 
  \begin{equation} 
 \lim\limits_{j \to \infty} d\bigl(S_1(x_{n_j}),y\bigr)  =  \lim\limits_{j \to \infty} d(x_{n_j},y). 
 \end{equation}
That is $  \lim\limits_{j \to \infty} d\bigl(S_1(x_{n_j}),y\bigr)-  \lim\limits_{j \to \infty} d(x_{n_j},y) = 0.$
Since $S_1$ is strongly quasi-nonexpansive, then $  \lim\limits_{j \to \infty} d\bigl(S_1(x_{n_j}),x_{n_j}\bigr) = 0.$   
Since $S_1$ is $\Delta$-demiclosed, then the $\Delta$-limit of $\{ x_{n_j}\} $ is belong to $Fix(S_1)$.
This mean that $p \in Fix(S_1).$\\

 For case $i=2$, we will show that $p \in Fix(S_2).$ From equation (\ref{eq6}), we have
  \begin{equation} 
 \lim\limits_{j \to \infty} d\bigl(S_2(S_1(x_{n_j})),y\bigr)  =  \lim\limits_{j \to \infty} d\bigl(S_1(x_{n_j}),y \bigr).
 \end{equation}
Since $S_2$ is strongly quasi-nonexpansive, then $  \lim\limits_{j \to \infty} d\bigl(S_2 S_1(x_{n_j}) ,x_{n_j}\bigr) = 0 $. 
Since $S_2$ is $\Delta$-demiclosed, then the $\Delta$-limit of $\bigl\{ S_1\bigl(x_{n_j}\bigr) \bigr\} $ is belong to $Fix(S_2)$.\\

Next, to show that $\bigl\{ S_1\bigl(x_{n_j}\bigr)\bigr\} $ is a $\Delta$-convergent to $p$ (that is to show $p \in Fix(S_2)$ ). 
Let $\{ S_1\bigl(x_{n_{j_r}}\bigr)\} $ be any subsequence of $\{ S_1\bigl(x_{n_j}\bigr)\} .$  Then\\

\noindent
\begin{align}
 \limsup\limits_{r \to \infty} d\bigl(S_1(x_{n_{j_r}}),p\bigr)
 &\leq  \limsup\limits_{r \to \infty} d\bigl(x_{n_{j_r}}),p\bigr) \nonumber\\
 &\leq  \limsup\limits_{r \to \infty} d\bigl(x_{n_{j_r}}),t\bigr),  \quad\forall t \in X\nonumber\\\
  &\leq  \limsup\limits_{r \to \infty} \Bigr[ d\bigl(x_{n_{j_r}}),S_1(x_{n_{j_r}})\bigr) + d\bigl(S_1(x_{n_{j_r}}),t\bigr)\Bigr], \quad\forall t \in X\nonumber\\
 &\leq  \limsup\limits_{r \to \infty}  d\bigl(x_{n_{j_r}}),S_1(x_{n_{j_r}})\bigr) + \limsup\limits_{r \to \infty} d\bigl(S_1(x_{n_{j_r}}),t\bigr),   \quad\forall t \in X\nonumber\\
&  \leq  \limsup\limits_{r \to \infty} d\bigl(S_1(x_{n_{j_r}}),t\bigr), \quad\forall t \in X.
\end{align}\\
This means that $\mathcal{A} \Big( \big\{ S_1 (x_{n_{j_r}})\big\}\Big)  = \{ p \}$  for each subsequence $\big\{ S_1 (x_{n_{j_r}})\big\}$ of $\big\{ S_1 (x_{n_j})\big\}.$  Thus $\big\{ S_1 (x_{n_j})\big\}$ is a $\Delta$-convergent to $p$, Therefore, $p \in Fix(S_2).$ \\

Next, in case $i = \{3, ..., m\}.$  First, we suppose that $p \in Fix(S)$ holds for $i = \{1, ..., l-1\}$ and let us consider $l =\{3, ..., m\}.$  We would like to show that $p \in Fix(S_l)$. 
By using the condition (iii), then we obtain the equation (\ref{eq4}), and equation (\ref{eq6}). Then we obtain 
\[ \lim\limits_{j \to \infty} d\bigl(S_{l}  S_{l-1} \cdots S_1(x_{n_j}),y \bigr) = \lim\limits_{j \to \infty} d\bigl(S_{l-1} \cdots S_1(x_{n_j}),y\bigr) < \infty.\]
Since $S_l$ is strongly quasi-nonexpansive, then we have  
\[
  \lim\limits_{j \to \infty} d\Bigl(S_{l} \bigl( S_{l-1} \cdots S_1(x_{n_j})\bigr), S_{l-1} \cdots S_1(x_{n_j})\Bigr) = 0.   
\]
and by $S_l$ is $\Delta$-demicolsed, the $\Delta$-limit of $\Bigr\{ S_{l-1} \cdots S_1 \bigl(x_{n_j}\bigr)\Bigr\} $ is belong to $Fix(S_l)$.\\

Next, we will show that $\Bigr\{ S_{l-1} \cdots S_1 \bigl(x_{n_j}\bigr)\Bigr\} $ is $\Delta$-convergent to $p$. \\
Let  $\Bigr\{ S_{l-1} \cdots S_1 \bigl(x_{n_{j_r}}\bigr)\Bigr\} $ be a subsequence of  $\Bigr\{ S_{l-1} \cdots S_1 \bigl(x_{n_j}\bigr)\Bigr\} .$ Then we have
\begin{align}
 \limsup\limits_{r \to \infty} d\bigl(S_{l-1} \cdots S_1(x_{n_{j_r}}),p\bigr)&~\leq  \limsup\limits_{r \to \infty} d\bigl(S_{l-2} \cdots S_1(x_{n_{j_r}}),p\bigr) \nonumber\\
 &~\leq  \limsup\limits_{r \to \infty} d\bigl(S_{l-2} \cdots S_1(x_{n_{j_r}}),t\bigr), \quad\forall t \in X\nonumber\\
  &~\leq  \limsup\limits_{r \to \infty} \Bigr[ d\bigl(S_{l-2} \cdots S_1(x_{n_{j_r}}),S_{l-1} \cdots S_1(x_{n_{j_r}})\bigr)\nonumber\\
  &~ \quad \quad + d\bigl(S_{l-1} \cdots S_1(x_{n_{j_r}}),t\bigr)\Bigr], \, \quad\forall t \in X\nonumber\\
  &~\leq  \limsup\limits_{r \to \infty} d\bigl(S_{l-2} \cdots S_1(x_{n_{j_r}}),S_{l-1} \cdots S_1(x_{n_{j_r}})\bigr)\nonumber\\
  &~ \quad \quad + \limsup\limits_{r \to \infty} d\bigl(S_{l-1} \cdots S_1(x_{n_{j_r}}),t\bigr), \,\quad\forall t \in X
\end{align}
By the assumption holds for $l-1$, then we have
\[   \lim\limits_{j \to \infty} d \bigl( S_{l-1} \cdots S_1(x_{n_j}), S_{l-2} \cdots S_1(x_{n_j})\bigr) = 0.\]
Then, we obtain
\[ \limsup\limits_{r \to \infty} d\bigl(S_{l-1} \cdots S_1(x_{n_{j_r}}),p\bigr) 
 \leq  \limsup\limits_{r \to \infty} d\bigl(S_{l-1} \cdots S_1(x_{n_{j_r}}),t\bigr), \]
 for all $  t \in X.$ Which means that $\mathcal{A} \Big( \big\{ S_{l-1} \cdots S_1 (x_{n_{j_r}})\big\}\Big)  = \{ p \}$  for each subsequence $\big\{ S_{l-1} \cdots  S_1 (x_{n_{j_r}})\big\}$ of $\big\{ S_{l-1} \cdots S_1 (x_{n_j})\big\}.$  
This implies that the sequence $\big\{ S_{l-1} \cdots S_1 (x_{n_j})\big\}$ is a $\Delta$-convergent to $p$. Therefore,  $p \in Fix(S_l).$ 
For all cases mentioned above, and $Fix(S) = \bigcap\limits_{i=1}^m Fix(S_i),$  this implies that $p \in \bigcap\limits_{i=1}^m Fix(S_i).$ In other words, $p \in Fix(S).$ Therefore, $S$ is an orbital $\Delta$-demiclosed mapping.
\end{proof}

\vspace{2mm}
Next, we establish the $\Delta$-convergence theorem for a sequence generated by a strongly quasi-nonexpansive orbital $\Delta$-demiclosed mapping in Hadamard spaces.\\


\begin{theorem}\label{thm1}
Let $(X, d) $ be a Hadamard space. If $ S: X \to X $ is a strongly quasi-nonexpansive orbital $\Delta$-demiclosed mapping, and $Fix(S) \neq \emptyset.$ If the sequence defined by $x_n := S^n(p_0)$ for each $n \in \mathbb{N}$, and for any initial point $p_0 \in X$.  Then, the sequence $x_n$ is $\Delta$-convergent to an element  $p^\ast \in Fix(S).$
\end{theorem}

\begin{proof}
Given, $p \in X$, and consider the sequence $\{ x_n\}$ which defined by $x_1 := p$ and $x_{n+1} := S^n(p) := S(x_n)$ for all $n \in \mathbb{N}$. 
    
    Let $y \in Fix(S)$, and since the mapping $S$ is quasi-nonexpansive. Then we obtain, 
\[
    0 \leq d(x_{n+1},y) \leq d(x_n,y) \leq \cdots \leq d(x_1,y) = d(p, y) < \infty. 
\]
    We can see that the sequence $\{d(x_n, y)\}$ is bounded and non-increasing for each $y \in Fix(S).$  Then the sequence $\{d(x_n, y)\}$ is convergent to an element of $[0, \infty)$, for each $y \in Fix(S).$  Take limit as $n \to \infty$, and by $x_{n+1}:= S(x_n)$. Then we obtain  
\[
    0 \leq  \lim\limits_{n \to \infty} d\left(S(x_n),y\right)  = \lim\limits_{n \to \infty} d(x_n,y) \leq \lim\limits_{n \to \infty} d(p, y) = d(p, y) < \infty. 
\]
Since the mapping $ S $ is strongly quasi-nonexpansive, we have
    \begin{equation} \label{eq-thm-1}\vspace{2mm}
    \lim\limits_{n \to \infty} d\bigl(x_n,S(x_n)\bigr) = 0
    \end{equation}
    
Since the sequence $\{ x_n\}$ is bounded, by Lemma \ref{kirk-lemma}, the sequence $\{ x_n\}$  has a $\Delta$-convergent subsequence. 
Now, let us take $p^\ast \in \omega_\Delta \bigl( \{ x_n\} \bigr)$ such that  $x_{n_k}$ is  $\Delta$-convergent to $p^\ast$, for some subsequence $\{ x_{n_k} \}$ of $\{ x_n\}$. 
By the assumption of $\{ x_{k+1}\}$, we have $ x_{n_k} = S^{n_k-1}$ is a subsequence. Then it follow that, by equation (\ref{eq-thm-1}), 
we have 
\[
    \lim\limits_{k \to \infty} d\bigl(x_{n_k},S(x_{n_k})\bigr) = 0. 
\]
 Since $ S $ is orbital $\Delta$-demiclosed mapping, we obtain $p^\ast$ is belong to $ Fix(S)$. Thus, $ \omega_\Delta \bigl( \{ x_n\} \bigr) \subset Fix(S)$. This implies that the sequence $\{ d(x_n,p^\ast) \}$ convergent to an element of $[0, \infty),~$ for all $p^\ast \in \omega_\Delta \bigl( \{ x_n\} \bigr)$. 
By Lemma \ref{dhompongsa-lemma}, the sequence $\{ x_n\}$ is $\Delta$-convergent to an element of $X$, and we have $ \omega_\Delta \bigl( \{ x_n\} \bigr) \subset Fix(S)$ from above. 
    Therefore, the sequence $\{ x_n\}$ is $\Delta$-convergent to an element of $Fix(S)$. This completes the prove. 
\end{proof}

\begin{corollary}\label{cor3.1}
Let $(X, d) $ be a Hadamard space. Suppose that $ S_i: X \to X $ is a strongly quasi-nonexpansive $\Delta$-demiclosed mapping, for each $i \in \{ 1, 2, ..., m\}$ and  $\bigcap\limits_{i=1}^m Fix(S_i) \neq \emptyset$. If $ S=S_m S_{m-1} \cdots S_1 $ and $x_n := S^n(p_0)$ for each $n \in \mathbb{N}$, where $p_0 \in X$,  then the sequence $\{x_n\}$ is $\Delta$-convergent to an element  $p^\ast \in  \bigcap\limits_{i=1}^m Fix(S_i).$
\end{corollary}

\begin{proof}
By using Proposition \ref{prop1}, and Theorem \ref{thm1}, we can conclude the claim.
\end{proof}

\section{Applications}

In this section, we consider some practical applications of our main results. We will apply our results from the previous section to solve convex optimizations on Hadamard spaces.

\subsection{Convex Minimization Problems} 
Firstly, we will present the application of solving convex minimization problems of proper lower semi-continuous (lsc) convex functions in a Hadamard space $X$. Let $f:X \rightarrow (-\infty,\infty]$ be a proper lsc convex function and $\lambda > 0$. According to section 2, the resolvent mapping $R_{\lambda_{i},f}$ is a strongly quasi-nonexpansive $\Delta$-demiclosed mapping such that $Fix(R_{\lambda,f}) = \underset{X}{\argmin} f$ whenever $\underset{X}{\argmin} f \neq \emptyset$. We then obtain the following corollary.\\

\begin{corollary}\label{cor4.1}
    Let $(X, d)$ be a Hadamard space and $f: X \to ( -\infty, +\infty]$ be a proper lsc convex function. Let $R_{\lambda_i, f} : X \to X$ be a resolvent mapping of $f$ with respect to $\lambda_i > 0$ for all  $i \in \{ 1, 2, ..., m\}.$ 
    If  $\argmin_X f \neq \emptyset$, 
    then for any initial point  $x_0 \in X$, the sequence defined for each $n \in \mathbb{N}$ by
    \[
    x_n := (R_{\lambda_m, f} R_{\lambda_{m-1}, f} \cdots R_{\lambda_1, f} )^n(x_0)
    \]
is $\Delta$-convergent to an element $x^\ast$  in $\argmin_X f.$
\end{corollary}

\begin{proof}
    Since $R_{\lambda_i, f}$ is a strongly quasi-nonexpansive $\Delta$-demiclosed mapping for all  $i \in \{ 1, 2, ..., m\}$ such that  $\bigcap\limits_{i=1}^m  Fix \bigl( R_{\lambda_i, f} \bigr) = Fix \bigl( R_{\lambda_i, f} \bigr) = \argmin_X f \neq \emptyset $, then we have that $R_{\lambda_m, f} R_{\lambda_{m-1}, f} \cdots R_{\lambda_1, f}$ is a strongly quasi-nonexpansive orbital $\Delta$-demiclosed mapping by using Proposition  \ref{prop1}. 
    Let $x_0 \in X$, define a sequence $\{ x_n \} \in X $ by \[x_n = (R_{\lambda_m, f} R_{\lambda_{m-1}, f} \cdots R_{\lambda_1, f} )^n(x_0).\] Then the sequence $\{ x_n \} $ is $\Delta$-convergent to some point in $x^\ast \in Fix(R_{\lambda_m, f} R_{\lambda_{m-1}, f} \cdots R_{\lambda_1, f} )$, which $Fix(R_{\lambda_m, f} R_{\lambda_{m-1}, f} \cdots R_{\lambda_1, f} )  = \bigcap\limits_{i=1}^m  Fix \bigl( R_{\lambda_i, f} \bigr)$ by using Theorem \ref{thm1}. 
    Therefore, the point $x^\ast \in \argmin_X f$.
\end{proof}

In the following application, we will examine the sum of multiple loss functions. Given the objective function $f$ is defined as $f := \sum_{i=1}^m f_i$, where $f_i$ is a proper lsc convex function for all $i \in \{ 1,2,...,m \}$. Since the resolvent mapping $R_{\lambda_{i},f}$ is a strongly quasi-nonexpansive $\Delta$-demiclosed mapping such that $Fix(R_{\lambda_{i},f}) = \underset{X}{\argmin} f$ for all $i \in \{ 1,2,...,m \}$. By Lemma \ref{lem1}, we have $\bigcap\limits_{i=1}^m  Fix(R_{\lambda_{i},f_i}) = Fix (R_{\lambda_{m},f_m} R_{\lambda_{m-1},f_{m-1}} \cdots R_{\lambda_{1},f_1})$, and $\bigcap\limits_{i=1}^m \underset{X}{\argmin}f_{i} = \underset{X}{\argmin} f$ whenever $\bigcap\limits_{i=1}^m \underset{X}{\argmin}f_{i}$ is nonempty set. Next, we can apply the following corollary to find the solutions of the objective function.\\

\begin{corollary}\label{cor4.2}
    Let $(X, d)$ be a Hadamard space and $f: X \to ( -\infty, +\infty]$ defined by $f:= \sum_{i=1}^m f_i$ where $f_i$ is a proper lsc convex function for all  $i \in \{ 1, 2, ..., m\}.$ . Let $R_{\lambda_i, f_i} : X \to X$ be a resolvent operator of $f_i$ with respect to $\lambda_i > 0$ for all  $i \in \{ 1, 2, ..., m\}.$ 
    If  $\bigcap\limits_{i=1}^m  \argmin_X f_i \neq \emptyset$, then for any initial point  $x_0 \in X$, the sequence defined for each $n \in \mathbb{N}$ by 
    \[    x_n := (R_{\lambda_m, f_m} R_{\lambda_{m-1}, f_{m-1}} \cdots R_{\lambda_1, f_1} )^n(x_0)    \]
is $\Delta$-convergent to an element $x^\ast$ in $\argmin_X f.$
\end{corollary}

\begin{proof}
    Since $R_{\lambda_i, f_i}$ is a strongly quasi-nonexpansive $\Delta$-demiclosed mapping for all  $i \in \{ 1, 2, ..., m\}$ such that $\bigcap\limits_{i=1}^m  Fix \bigl( R_{\lambda_i, f_i} \bigr) = \bigcap\limits_{i=1}^m \argmin_X f_i \neq \emptyset $, then we have that \[R_{\lambda_m, f_m} R_{\lambda_{m-1}, f_{m-1}} \cdots R_{\lambda_1, f_1}\] is a strongly quasi-nonexpansive orbital $\Delta$-demiclosed mapping by using Proposition  \ref{prop1}. 
    Let $x_0 \in X$ and define a sequence $\{ x_n \} \in X $ by assumption. 
    Then the sequence $\{ x_n \} $ is $\Delta$-convergent to some point $x^\ast$  which is $x^\ast \in Fix(R_{\lambda_m, f_m} R_{\lambda_{m-1}, f_{m-1}} \cdots R_{\lambda_1, f_1} ) = \bigcap\limits_{i=1}^m  Fix \bigl( R_{\lambda_i, f_i} \bigr)$ by using Theorem \ref{thm1}. 
    Therefore, the point $x^\ast \in \bigcap\limits_{i=1}^m \argmin_X f_i$. \\
    As $\bigcap\limits_{i=1}^m \underset{X}{\argmin}f_{i} \neq \emptyset$ by assumption, then $\bigcap\limits_{i=1}^m \underset{X}{\argmin}f_{i} = \underset{X}{\argmin} f$. 
    Hence, the point $x^\ast \in \argmin_X f$.  This completes the prove.   
\end{proof}

\vspace{2mm}
\subsection{Convex Feasibility Problem}

We consider the feasibility problem with each $C_{i}$ being closed and convex. We call this particular case the convex feasibility problem. Let $(X,d)$ be a Hadamard space and let $C_1,C_2,...,C_m$ be nonempty closed convex subsets of $X$. The convex feasibility problem is to find some point
\[ x \in \bigcap\limits_{i=1}^m C_i , \]
when this intersection is nonempty. According to section 2, we know that the metric projection $P_{C_{i}}$ from $X$ onto $C_{i}$ is a strongly quasi-nonexpansive $\Delta$-demiclosed mapping, and $Fix(P_{C_{i}})= C_{i}$ for all $i \in \{ 1,2,...,m \}$. If $\bigcap\limits_{i=1}^{m} C_{i}$ is nonempty, we can use the alternating projection mapping to find the solution to this problem as follows:
\vspace{2mm}

\begin{corollary}\label{cor4.3}
    Let $(X, d)$ be a Hadamard space. 
    If $P_{C_i}$ is the metric projection of $X$ onto nonempty closed and convex subset $C_i \subset X$ for all  $i \in \{ 1, 2, ..., m\}$ with  $\bigcap\limits_{i=1}^m  C_i \neq \emptyset$, 
    then for any initial point  $x_0 \in X$, the sequence defined for each $n \in \mathbb{N}$ by
    \[
    x_n := (P_{C_m} P_{C_{m-1}} \cdots P_{C_1} )^n (x_0)
    \]
    is $\Delta$-convergent to an element $x^\ast$ in $\bigcap\limits_{i=1}^m  C_i$.
\end{corollary}

\begin{proof}
    Let each $P_{C_i}$ is a strongly quasi-nonexpansive $\Delta$-demiclosed mapping for all $i \in \{ 1, 2, ..., m\}$. 
    By Proposition \ref{prop1}, the product $P_{C_m} \cdots P_{C_1}$ is a strongly quasi-nonexpansive orbital $\Delta$-demiclosed mapping. 
    By Theorem \ref{thm1}, for each $x \in X,$ there exists  $x^\ast = x^\ast(x) \in Fix(P_{C_m} \cdots P_{C_1})$ such that the sequence $\{ (P_{C_m} P_{C_{m-1}} \cdots P_{C_1} )^n (x) \}$ is $\Delta$-convergent to $x^\ast$. 
    Hence, by Lemma \ref{lem1}, we can conclude that  $x^\ast \in \bigcap\limits_{i=1}^m  C_i.$
\end{proof}
Next, we prove the strong convergence alternating projection from  the previous corollary by letting $C_j$ is a compact set for some $j \in \{ 1,2,...,m \}$.
\vspace{3mm}

\begin{corollary}\label{cor4.4}
    Let $(X, d)$ be a Hadamard space. 
    If $P_{C_i}$ is the metric projection of $X$ onto nonempty closed and convex subset $C_i \subset X$ for all  $i \in \{ 1, 2, ..., m\}$ with  $\bigcap\limits_{i=1}^m  C_i \neq \emptyset$, and $C_j$ is a compact set for some  $j \in \{ 1, 2, ..., m\}$
    then for any initial point  $x_0 \in X$, the sequence defined for each $n \in \mathbb{N}$ by    \[    x_n := (P_{C_m} P_{C_{m-1}} \cdots P_{C_1} )^n (x_0)    \]
    is convergent to an element $x^\ast$ in $\bigcap\limits_{i=1}^m  C_i$.    
\end{corollary}

\begin{proof}
    First, We will divide into two cases of $C_i$ be a compact set for some $i \in \{1, . . . ,m\}$.

Case I: Suppose that $C_m$ is a compact set. 
Let the starting point $x_0 = x$ for some element $x \in X$ and $x_n := (P_{C_m} P_{C_{m-1}} \cdots P_{C_1} )^n (x)$ for all $n \in \mathbb{N}$. By Corollary \ref{cor4.3}, we have that the sequence $\{x_n\}$ is $\Delta$-convergent to $x^\ast$ for some $x^\ast(x) \in \bigcap\limits_{i=1}^m  C_i$.
Then, every subsequence $\{x_{n_k}\}$ of $\{x_n\}$ is $\Delta$-convergent to $x^\ast$. Since $C_m$ is a compact set and $\{x_n\} \subset C_m$. Then for every subsequence $\{x_{n_k}\}$  of $\{x_n\}$ there exists $\{x_{n_{k_l}}\}$ subsequence of $\{x_{n_k}\}$ such that $\lim\limits_{l \to \infty} x_{n_{k_l}} = x^\ast$, this implies that $\lim\limits_{n \to \infty} x_{n} = x^\ast$.

Case II: Suppose that $C_i$ is a compact set for some $i \in \{1, . . . ,m-1\}$.
Next we define $\hat{P}:= P_{C_i} P_{C_{i-1}} \cdots P_{C_1} P_{C_m} P_{C_{m-1}} \cdots P_{C_{i+1}} $ 
and $P :=P_{C_m} P_{C_{m-1}} \cdots P_{C_1}$. By case I, we have
$\lim\limits_{n \to \infty} (\hat{P})^n(P_{C_i} P_{C_{i-1}} \cdots P_{C_1}) (x)= x^\ast$ which  $x^\ast \in \bigcap\limits_{i=1}^m  C_i$ for any  $x \in X.$
Observe that 
\begin{align*}
  P^n(x) 
  &= (P_{C_m} P_{C_{m-1}} \cdots P_{C_1})^n(x) \\
  &= (P_{C_m} P_{C_{m-1}} \cdots P_{C_{i+1}})
(\hat{P})^{n-1}(P_{C_i} P_{C_{i-1}} \cdots P_{C_1} )(x)
\end{align*}
and then $d(P^n(x), x^\ast) \leq d((\hat{P})^n(P_{C_i}P_{C_{i-1}} \cdots P_{C_1} )(x), x^\ast) \to 0$ as $n \to \infty$. 
Therefore, we can conclude that $\lim\limits_{n \to \infty}
(P_{C_m} P_{C_{m-1}} \cdots P_{C_1} )^n(x) = x^\ast.$
\end{proof}

\section{Conclusions}

Let $(X,d)$ be a Hadamard space, and let $S_i : X \to X$ be a strongly quasi-nonexpansive mapping for all $i \in \{1, 2, . . . ,m\}$ such that $\bigcap\limits_{i=1}^m  Fix (S_i) \neq \emptyset$. Firstly, in Lemma (\ref{lem1}) we prove that $ Fix (S_m S_{m-1} \cdots S_1) = \bigcap\limits_{i=1}^m  Fix (S_i)$. Next, we have shown that the product of strongly quasi-nonexpansive $\Delta$-demiclosed mappings is a strongly quasi-nonexpansive orbital $\Delta$-demiclosed mapping in Hadamard spaces (Proposition \ref{prop1}). We have also proved the $\Delta$-convergence theorem for approximating a common fixed point of infinite products of strongly quasi-nonexpansive orbital $\Delta$-demiclosed mappings in Hadamard spaces (Theorem \ref{thm1}). Our results have some applications in convex function minimization (Corollary \ref{cor4.1}), the minimization of the sum of finitely many convex functions (Corollary \ref{cor4.2}), and the convex feasibility problem for finitely many sets (Corollary \ref{cor4.3} and \ref{cor4.4}) in Hadamard spaces.

\section*{Acknowledgments}

The authors acknowledge the financial support provided by the Center of Excellence in Theoretical and Computational Science (TaCS-CoE), King Mongkut's University of Technology Thonburi.
The first author was supported by the Petchra Pra Jom Klao Ph.D. Research Scholarship (Grant No. 51/2562) from KMUTT.

\end{document}